\documentclass[draft]{amsart}
\usepackage{amsmath}
\usepackage{amssymb}

\newtheorem{theorem}{Theorem}[section]
\newtheorem{lemma}[theorem]{Lemma}
\newtheorem{definition}[theorem]{Definition}
\newtheorem{corollary}[theorem]{Corollary}

\theoremstyle{remark}
\newtheorem{remark}{Remark}
\newtheorem{example}{Example}

\numberwithin{theorem}{section}
\numberwithin{equation}{section}

\newcommand{\cv}{\mathbb{C}}

\newcommand{\dv}{\mathbb{D}}

\newcommand{\aut}{\textup{Aut}}
\newcommand{\ds}{\displaystyle}
\def\Re{{\sf Re}}

\begin{document}

\title[Generalized squeezing functions and Fridman invariants]{On the generalized squeezing functions and Fridman invariants of special domains}

\author[F. Rong, S. Yang]{Feng Rong, Shichao Yang}

\address{School of Mathematical Sciences, Shanghai Jiao Tong University, 800 Dong Chuan Road, Shanghai, 200240, P.R. China}
\email{frong@sjtu.edu.cn}

\address{School of Mathematics and Statistics, Huizhou University, 46 Yan Da Boulevard, Huizhou, Guangdong, 516007, P.R. China}
\email{yangshichao68@163.com}



\thanks{The authors are partially supported by the National Natural Science Foundation of China (grant no. 11871333).}

\begin{abstract}
The main purpose of this paper is to study the generalized squeezing functions and Fridman invariants of some special domains. As applications, we give the precise form of generalized squeezing functions and Fridman invariants of various domains such as $n$-dimensional annuli. Furthermore, we provide domains with non-plurisubharmonic generalized squeezing function or Fridman invariant.
\end{abstract}

\maketitle

\section{Introduction}

To study complex and geometric structures of a domain, one may consider holomorphic maps from some standard domains such as balls to this domain and vice visa.

Let $D$ be a bounded domain and $\Omega$ a bounded homogeneous domain in $\mathbb {C}^n$ and $z\in D$. Denote by $O(\Omega,D)$ the set of holomorphic maps from $\Omega$ into $D$. Denote by $d$ either the Carath\'{e}odory pseudosdistance $c$ or the Kobayashi pseudodistance $k$ on $D$. Fridman \cite{Fridman1979} introduced a holomorphic invariant, now called the \textit{Fridman invariant}, as follows:
$$e^{\Omega^{d}}_{D}(z)= \sup\left\{\tanh(r):\ B^{d}_{D}(z,r) \subset f(\Omega),\ f\in O(\Omega, D),\ f\:\textup{is  injetive}\right\},$$
where $B^{d}_{D}(0, r)$ is the $d$-ball centered at $z$ with radius $r$ (in \cite{Fridman1979, Fridman1983}, $\inf \frac{1}{r}$ was used instead of $\sup \tanh(r)$). We denote $e^{\Omega^c}_{D}(z)$ by $\tilde{e}^{\Omega}_{D}(z)$ and $e^{\Omega^k}_{D}(z)$ by $e^{\Omega}_{D}(z)$ in this paper. When $\Omega$ is the unit ball $B^n$, $\tilde{e}^{B^n}_{D}(z)$ (resp. $e^{B^n}_{D}(z)$) is denoted by $\tilde{e}_{D}(z)$ (resp. $e_{D}(z)$).

Deng-Guan-Zhang \cite{Deng2012} introduced another invariant, called the \textit{squeezing function}, as follows:
$$s_{D}(z)=\sup\left\{r:\ rB^{n} \subset f(D),\ f\in O(D,B^{n}),\ f(z)=0,\ f\:\text{is injective}\right\}.$$

Let $\Omega$ be a bounded, balanced and convex domain in $\mathbb {C}^{n}$. Then $\Omega=\{z\in \mathbb {C}^{n}:\ \rho_{\Omega}(z)<1\}$, where $\rho_{\Omega}(z)$ is the Minkowski function of $\Omega$. It is known that $\rho_{\Omega}(z)$ is a $\mathbb{C}$-norm (see e.g. \cite[Lemma 3.3]{RY}). In \cite{RY}, we defined the \textit{generalized squeezing function} as follows:
$$s_D^\Omega(z)=\sup\{r:\ B^{\rho_{\Omega}}_{n}(0,r) \subset f(D),\ f\in O(D,\Omega),\ f(z)=0,\ f\: \text{is injective}\},$$
where $B^{\rho_{\Omega}}_{n}(0,r)=\{z\in \mathbb{C}^{n}:\ \rho_{\Omega}(z)<r\}$. It is easy to see that $s^{\mathbb{B}^n}_{D}(z)=s_{D}(z)$.

From the definitions, it is clear that $e^{\Omega^{d}}_{D}$ and $s_D^\Omega$ are invariant under biholomorphisms. Many properties and applications of $s_{D}$ have been recently explored by various authors (see e.g \cite{Deng2012, FR, JK, KZ2016, Nikolov2020, Zimmer2018, Zimmer2019} and the references therein). There have also been some studies on $\tilde{e}_{D}$ and $e_{D}$ (see e.g. \cite{Deng-Zhang2019, Fridman1983, MV2012, Nikolov-Verma2018}). For more details on various recent results, we refer the readers to the survey papers \cite{Dengsurvey2019, Zhang2017}.

It is difficult to compute the squeezing function, and there are very few domains on which the precise form of the squeezing function is known. Deng-Guan-Zhang \cite{Deng2012} showed that the squeezing functions of classical bounded symmetric domains are certain constants, and that $s_{B^{n}\backslash \{0\}}(z)=\|z\|$. Ng-Tang-Tsai \cite{NTT} have shown that $s_{A_{r}}(z)=\max \{|z|, \frac{r}{|z|}\}$, where $A_{r}=\{z \in \mathbb{C}:\ 0<r<|z|<1\}$. The original motivation of this paper is to give the precise form of the squeezing functions for the $n$-dimensional annuli $A^{n}_{r}=\{z \in \mathbb{C}^{n}:\ 0<r<\|z\|<1\}$, $n\ge 2$.

On the other hand, Forn\ae ss-Shcherbina \cite{FS} have proved that there exist strictly pseudoconvex domains in $\mathbb{C}^{2}$ whose squeezing function is not plurisubharmonic. It is natural to consider domains on which the generalized squeezing function or the Fridman invariant is not plurisubharmonic.

The paper is organized as follows. In section \ref{S:invariants}, the generalized squeezing functions and Fridman invariants of some special domains are studied in terms of the Carath\'{e}odory pseudodistance. In section \ref{S:examples}, the precise form of the generalized squeezing functions and Fridman invariants of some example domains are given, including that of the $n$-dimensional annuli. In section \ref{S:nonpsh}, domains with non-plurisubharmonic generalized squeezing function or Fridman invariant are given.

\section{Squeezing functions and Fridman invariants of special domains}\label{S:invariants}

The first type of special domains we consider is constructed by deleting compact subsets from bounded, balanced, convex and homogeneous domains in $\mathbb{C}^n$.

Let $\mathbb D$ be the unit disk in $\cv$. Recall that the Carath\'{e}odory pseudodistance on $\Omega$ is defined as
$$c_\Omega (z,w)=\sup\left\{\tanh^{-1}|\lambda|:\ f\in O(\Omega,\mathbb D),\ f(z)=0,\ f(w)=\lambda \right\}.$$
Let $\Omega$ be a bounded, balanced, convex and homogeneous domain in $\mathbb{C}^n$, and $S\subset \Omega$ a compact subset. Denote
$$d_{c_{\Omega}}^{S}(z)=\min_{w\in S}\tanh[c_{\Omega}(z,w)],\ \ \ z\in \Omega\backslash S.$$

\begin{theorem}\label{KB}
Let $\Omega$ be a bounded, balanced, convex and homogeneous domain in $\mathbb{C}^n$, $n\ge 2$. If $K$ is a compact subset of $\Omega$ and $\Omega\backslash K$ is connected, then we have
$$s^{\Omega}_{\Omega \backslash K}(z)=d_{c_{\Omega}}^{\partial K}(z)=d_{c_{\Omega}}^K(z),\ \ \ z\in \Omega \backslash K.$$
\end{theorem}

For the proof of Theorem \ref{KB}, we need the following three lemmas.

\begin{lemma}\cite[Theorem 1]{Barth}\label{snorm}
Let $\Omega$ be a balanced and convex domain in $\Bbb C^{n}$. Then the Minkowski function $\rho_{\Omega} (z)$ of $\Omega$ is plurisubharmonic.
\end{lemma}

\begin{lemma}\cite[Proposition 2.3.1 (c)]{Pflug2013}\label{lnc}
Let $\Omega$ be a balanced and convex domain in $\Bbb C^{n}$. Then for all $z\in \Omega$, we have $\rho_\Omega(z)=\tanh[c_\Omega(0,z)]$.
\end{lemma}

The following is the well-known Hartogs's extension theorem (see e.g. \cite[Theorem 1.2.6]{Krantz}).

\begin{lemma}\label{Hartogs}
Let $\Omega$ be a domain in $\Bbb C^{n}$ and $K$ a compact subset of $\Omega$ such that $\Omega\backslash K$ is connected. If $f$ is holomorphic on $\Omega\backslash K$, then there exists a holomorphic function $F$ on $\Omega$ such that $F|_{\Omega\backslash K}=f$.
\end{lemma}

\begin{proof}[\textbf{Proof of Theorem \ref{KB}}]
Let $z\in \Omega \backslash K$. Since $\Omega$ is homogeneous, there exists $\psi \in \aut(\Omega)$ such that $\psi(z)=0$. Then $c_{\Omega}(z,w)=c_{\psi(\Omega)}(\psi(z),\psi(w))=c_{\Omega}(0,\psi(w))$.

By Lemma \ref{lnc}, we have $\rho_{\Omega}(w)=\tanh c_{\Omega}(0,\psi(w))$. Since $\Omega$ is convex, we have
$$\left\{tz+(1-t)w:\ 0\leq t\leq 1,\ z\in \Omega\backslash K,\ w \in \mathring{K} \right\} \cap \partial{K} \neq \emptyset.$$
This implies that $d_{c_{\Omega}}^{\partial K}(v)=d_{c_{\Omega}}^{K}(v)$, for all $v\in \Omega \backslash K$. Therefore,
$$B^{\rho_{\Omega}}_{n}(0,d_{c_{\Omega}}^{\partial K}(z))\subset \psi(\Omega),$$
where $B^{\rho_{\Omega}}_{n}(0,d_{c_{\Omega}}^{\partial K}(z))=\{v\in \mathbb{C}^{n}: \rho_{\Omega}(v)<d_{c_{\Omega}}^{\partial K}(z)\}$. It follows that $s^{\Omega}_{\Omega\backslash K}(z)\ge d_{c_{\Omega}}^{\partial K}(z)$.

Next we show that $s^{\Omega}_{\Omega\backslash K}(z)\leq d_{c_{\Omega}}^{\partial K}(z)$. Let $f:\Omega\backslash K  \rightarrow \Omega$ be a holomorphic embedding such that $f(z)=0$. By Lemma \ref{Hartogs}, there exists a holomorphic mapping $F:\Omega  \rightarrow \mathbb{C}^n$ such that $\left.F\right|_{\Omega \backslash K}=f$.

By Lemma \ref{snorm}, we have that $\rho_{\Omega}(z)$ is plurisubharmonic. Hence $\rho_{\Omega}(F(z))$ is a plurisubharmonic function on $\Omega$. By the maximum modulus principle for plurisubharmonic functions, we have $\rho_{\Omega}(F(K))<1$. Hence $F(\Omega)\subset \Omega$.

Since $f$ is injective holomorphic, it is easy to see that $F(\partial K)\cap F(\Omega \backslash K)=\emptyset$. By the decreasing property of the Catath\'{e}odory pseudodistance, we have
$$c_{\Omega}(0,F(\partial K))=c_{\Omega}(F(z),F(\partial K))\leq c_{\Omega}(z,\partial K).$$
This implies that
$$s^{\Omega}_{\Omega\backslash K}(z)\leq d_{c_{\Omega}}^{\partial K}(z).$$
Hence $s^{\Omega}_{\Omega\backslash K}(z)=d_{c_{\Omega}}^{\partial K}(z)$.
\end{proof}

Recently, Nikolov-Verma \cite{Nikolov-Verma2018} have proved that $s_{D}(z)\leq \tilde{e}_{D}(z)\leq e_{D}(z)$. This still holds for the generalized squeezing function (cf. \cite[Theorem 4.1]{RY}).

\begin{lemma}\label{es}
Let $\Omega$ be a bounded, balanced, convex and homogeneous domain and $D$ a bounded domain in $\mathbb{C}^n$. Then, we have
$$s_D^\Omega(z)\leq \tilde{e}^{\Omega}_{D}(z)\leq e^{\Omega}_{D}(z),\ \ \ z\in D.$$
\end{lemma}
\begin{proof}
Since $c_{D}\leq k_{D}$, it is easy to see that $\tilde{e}^{\Omega}_{D}(z)\leq e^{\Omega}_{D}(z)$. Now it suffices to prove that $s_D^\Omega(z)\leq \tilde{e}^{\Omega}_{D}(z)$.

Let $z\in D$ and $0<\epsilon<s_D^\Omega(z)$. Then, there exists a holomorphic embedding $f:D \rightarrow \Omega$ such that $f(z)=0$ and $B^{\rho_{\Omega}}_{n}(0,s_D^\Omega(z)-\epsilon) \subset f(D)$.

Set $g(w):=f^{-1}((s_D^\Omega(z)-\epsilon)w)$, which is an injective holomorphic mapping from $\Omega$ into $D$, with $g(0)=z$. By the decreasing property of the Carath\'{e}odory pseudodistance, we have
$$B^{c}_{f(D)}(0,\tanh^{-1}[s_D^\Omega(z)-\epsilon])\subset B^{c}_{\Omega}(0,\tanh^{-1}[s_D^\Omega(z)-\epsilon]).$$

By Lemma \ref{lnc}, we have
\begin{align*}
B^{c}_{\Omega}(0,\tanh^{-1}[s_D^\Omega(z)-\epsilon])
&=\{v:c_{\Omega}(0,v)<\tanh^{-1}[s_D^\Omega(z)-\epsilon]\} \\
&=\{v:\tanh[c_{\Omega}(0,v)]<s_D^\Omega(z)-\epsilon \} \\
&=\{v:\rho_{\Omega}(v)<s_D^\Omega(z)-\epsilon\} \\
&=B^{\rho_{\Omega}}_{n}(0,s_D^\Omega(z)-\epsilon)
\end{align*}
Thus, we get
\begin{align*}
B^{c}_{D}(z,\tanh^{-1}[s_D^\Omega(z)-\epsilon])&=f^{-1}(B^{c}_{f(D)}(z,\tanh^{-1}[s_D^\Omega(z)-\epsilon]))\\
&\subset f^{-1}(B^{\rho_{\Omega}}_{n}(0,s_D^\Omega(z)-\epsilon))=g(\Omega)
\end{align*}
This implies that $\tilde{e}^{\Omega}_{D}(z)\ge s_D^\Omega(z)-\epsilon$. Since $\epsilon$ is arbitrary, we have $\tilde{e}^{\Omega}_{D}(z)\ge s_D^\Omega(z)$.
\end{proof}

It is then natural and interesting to study when $s_D^\Omega(z)=\tilde{e}^{\Omega}_{D}(z)$. For domains $D=\Omega\backslash K$, where $K$ is a proper analytic subset or some special compact subset of $\Omega$, this is indeed the case.

\begin{theorem}\label{man}
Let $\Omega$ be a bounded, balanced, convex and homogeneous domain in $\mathbb{C}^n$, $n\ge 2$, and $G\subset \subset \Omega$ a subdomain such that  pseudoconvex  but not Levi flat points are dense in $\partial G$. If $D=\Omega \backslash \overline{G}$ is connected, then we have
$$s_D^\Omega(z)=\tilde{e}^{\Omega}_{D}(z)=d_{c_{\Omega}}^{\partial G}(z),\ \ \ z\in D.$$
\end{theorem}
\begin{proof}
By Theorem \ref{KB} and Lemma \ref{es}, it suffices to show that $\tilde{e}^{\Omega}_{D}(z)\leq d_{c_{\Omega}}^{\partial G}(z)$.

Let $P\in \partial G$ such that $d_{c_{\Omega}}^{\partial G}(z)=\tanh[c_{\Omega}(z,P)]$. Suppose that $\tilde{e}^{\Omega}_{D}(z)>d_{c_{\Omega}}^{\partial G}(z)$. Then, there exist $r>d_{c_{\Omega}}^{\partial G}(z)$ and a holomorphic embedding $f:\Omega \rightarrow D$ such that $f(0)=z$ and $B^{c}_{D}(z,r)\subset f(\Omega)$. By Lemma \ref{Hartogs}, we have $c_{D}(z_1,z_2)=c_{\Omega}(z_1,z_2)$, for all $z_1, z_2 \in D$.

Since the Carath\'{e}odory pseudodistance is continuous (see e.g. \cite{Pflug2013}), we know that $B^{c}_{D}(z,r)$ and $B^{c}_{\Omega}(z,r)$ are open. It follows that there exists $\delta>0$ such that $B^{n}(P,\delta)\subset B^{c}_{\Omega}(z,r)$. Since $B^{c}_{D}(z,r)\subset f(\Omega)\subset D$ and $c_{D}(z_1,z_2)=c_{\Omega}(z_1,z_2)$, we have $B^{n}(P,\delta)\cap  f(\Omega) \neq \emptyset$ and $B^{n}(P,\delta)\cap \partial G\subset \partial (f(\Omega))$. On the other hand, since pseudoconvex  but not Levi flat points are dense in $\partial G$, there exist a local $C^2$ defining function $\rho$ and $Q\in B^{n}(P,\delta)\cap \partial G$ such that
\[
\sum_{j, k=1}^{n} \frac{\partial^{2} \rho}{\partial z_{j} \partial \bar{z}_{k}}(Q) v_{j} \bar{v}_{k} > 0,
\]
for some $v \in \mathbb{C}^{n}$ satisfying
\[
\sum_{j=1}^{n} \frac{\partial \rho}{\partial z_{j}}(Q) v_{j}=0.
\]
However, $f(\Omega)$ is pseudoconvex and it is clear that $-\rho(z)$ is a local defining function on some neighborhood of $Q$ for $f(\Omega)$. It follows that
\[
\sum_{j, k=1}^{n} \frac{\partial^{2} (-\rho)}{\partial z_{j} \partial \bar{z}_{k}}(Q) v_{j} \bar{v}_{k} \ge 0,
\]
for all $v \in \mathbb{C}^{n}$ satisfying
\[
\sum_{j=1}^{n} \frac{\partial (-\rho)}{\partial z_{j}}(Q) v_{j}=0,
\]
which is a contradiction. Hence $\ds \tilde{e}^{\Omega}_{D}(z)\leq d_{c_{\Omega}}^{\partial G}(z)$.
\end{proof}

\begin{remark}
There are many bounded domains with dense pseudoconvex  but not Levi flat points in the boundary, such as strongly pseudoconvex domains, pseudoconvex domains with real analytic boundary, pseudoconvex domains of finite type, etc.
\end{remark}

The proof of the next result is similar to that of Theorem \ref{man}, and we only outline the main differences.

\begin{theorem}\label{bdary}
Let $\Omega$ be a bounded, balanced, convex and homogeneous domain in $\mathbb{C}^n, n\ge 2$, and $G\subset \subset \Omega$ a subdomain such that pseudoconvex  but not Levi flat  points are dense in $\partial G$. If $K\subset \partial G$ is an open subset and $D=\Omega \backslash \overline{K}$ is connected, then we have
$$s_D^\Omega(z)=\tilde{e}^{\Omega}_{D}(z)=d_{c_{\Omega}}^{\overline K}(z),\ \ \ z\in D.$$
\end{theorem}
\begin{proof}
By Theorem \ref{KB} and Lemma \ref{es}, it suffices to show that $\tilde{e}^{\Omega}_{D}(z)\leq d_{c_{\Omega}}^{\overline K}(z)$.

Let $P\in \overline{K}$ such that $d_{c_{\Omega}}^{\overline{K}}(z)=\tanh[c_{\Omega}(z,P)]$. Suppose that $\tilde{e}^{\Omega}_{D}(z)>d_{c_{\Omega}}^{\overline{K}}(z)$. Then, there exist $r>d_{c_{\Omega}}^{\overline K}(z)$ and a holomorphic embedding $f:\Omega \rightarrow D$ such that $f(0)=z$ and $B^{c}_{D}(z,r)\subset f(\Omega)$. Then, by Lemma \ref{Hartogs}, there exists $\delta>0$ such that $B^{n}(P,\delta)\backslash \overline{K} \subset f(\Omega)$.

Since the pseudoconvex  but not Levi flat  points are dense in $\partial G$, there exist a pseudoconvex  but not Levi flat  point $Q\in B^{n}(P,\delta)\cap K$ and $\delta_{1}>0$ such that $B^{n}(Q,\delta_{1})\cap \partial{G} \subset K$ and $B^{n}(Q,\delta_{1})\subset B^{n}(P,\delta)$. Let $G_{1}=B^{n}(Q,\delta_{1})\backslash\overline{G}$. It is clear that $G_{1}$ is a connected component of $B^{n}(Q,\delta_{1})\cap f(\Omega)$, thus pseudoconvex. Then the same argument as in the proof of Theorem \ref{man} gives a contradiction.
\end{proof}

Using similar ideas, replacing Hartogs's extension theorem by Riemann's removable singularity theorem, we also have the following result.

\begin{theorem}\label{ati}
Let $\Omega$ be a bounded, balanced, convex and homogeneous domain in $\mathbb{C}^n$, $n\ge 1$, and $H$ a proper analytic subset of $\Omega$. Let $A$ be a subset of $H$, and set $D=\Omega\backslash \overline A$. Then for any $z\in D$, we have
$$s^{\Omega}_D(z)=\tilde{e}^{\Omega}_D(z)=d_{c_{\Omega}}^{\overline A}(z).$$
Moreover, if $n\ge 2$ and $H$ is of codimension at least two, then
$$\tilde{e}^{\Omega}_D(z)=e^{\Omega}_D(z),\ \ \ z\in D.$$
\end{theorem}
\begin{proof}
Let $z\in \Omega\backslash H$. As in the proof of Theorem \ref{KB}, one readily checks that $s_D^\Omega (z)\ge d_{c_{\Omega}}^{\overline{A}}(z)$. Thus we only need to show that $\tilde{e}^{\Omega}_D(z)\leq d_{c_{\Omega}}^{\overline{A}}(z)$.

Suppose that $\tilde{e}^{\Omega}_D(z)>d_{c_{\Omega}}^{\overline{A}}(z)$. Then, there exist $r>d_{c_{\Omega}}^{\overline{A}}(z)$ and a holomorphic embedding $f:\Omega \rightarrow D$ such that $f(0)=z$ and $B^{c}_{D}(z,r)\subset f(\Omega)$.

Let $P\in \overline{A}$ such that $d_{c_{\Omega}}^{\overline{A}}(z)=\tanh[c_{\Omega}(z,P)]$. By Riemann removable singularity theorem, we have $c_{D}(w_1,w_2)=c_{\Omega}(w_1,w_2)$, for all $w_1,w_2 \in D$. Therefore, there exists $\delta>0$ such that $B^{n}(P,\delta)\backslash \overline{A} \subset f(\Omega)$.

Now consider $g=f^{-1}:B^{n}(P,\delta)\backslash \overline{A} \rightarrow \Omega$. By Riemann removable singularity theorem, there exists a holomorphic mapping $G:B^{n}(P,\delta)\rightarrow \overline{\Omega}$ such that $\left.G\right|_{B^{n}(P,\delta)\backslash \overline{A}}=g$. Since $\Omega$ is balanced and convex, by Lemma \ref{snorm}, $\rho_{\Omega}(z)$ is plurisubharmonic. Hence $\rho_{\Omega}(G(z))$ is plurisubharmonic on $B^{n}(P,\delta)$. Since $f$ is biholomorphic from $\Omega$ to $f(\Omega)$, it is easy to see that $G(B^{n}(P,\delta)\cap \overline{A}) \subset \overline{\Omega}$. Therefore $\rho_{\Omega}(G(z))=1$, for all $z\in B^{n}(P,\delta)\cap \overline{A}$. By the maximum modulus principle for plurisubharmonic functions, we have $\rho_{\Omega}(G(z))\equiv 1$, which is a contradiction.

When $n\ge 2$ and $H$ is of codimension at least two, by \cite[Corollary 3.4.3]{Pflug2013} and \cite[Theorem 1]{Lempert}, we have
$$k_{D}(w_1,w_2)=k_{\Omega}(w_1,w_2)=c_\Omega(w_1,w_2)=c_{D}(w_1,w_2),\ \ \ w_1,w_2 \in D.$$
This implies that $e^{\Omega}_{D}(z)=\tilde{e}^{\Omega}_{D}(z)$.
 \end{proof}

\begin{remark}
In Theorem \ref{ati}, letting $\Omega=\mathbb{D}$ and $A=\{0\}$, we get
$$s_{\mathbb{D} \backslash \{0\}}(z)=\tilde{e}_{\mathbb{D} \backslash \{0\}}(z)=|z|.$$
However $s_{\mathbb{D} \backslash \{0\}}(z)\not\equiv e_{\mathbb{D} \backslash \{0\}}(z)$ (\cite[Theorem 6]{RY1}).
\end{remark}

The remainder of this section is concerned with the squeezing function and Fridman invariant defined by the unit ball $B^n$.

\begin{definition}\label{def}
Let $D$ be a domain in $\mathbb{C}^{n}$, $n\ge 2$. A boundary point $P\in \partial{D}$ is said to have the \textbf{R-S property}, if there exists an open neighbourhood $U=U_{1}\times U_{2}$ of $P$ such that $U\cap D=D_{1}\times D_{2}$, where $U_{j}$, $j=1,2$, is an open subset of $\mathbb{C}^{n_{j}}$, $n_{1}+n_{2}=n$, $n_j>0$, and $D_{j}$ is a domain in $\mathbb{C}^{n_{j}}$, and one of the following conditions holds:\\
\textup{(i)} $\overline{D_{2}}\cap U_{2}\neq U_{2}$.\\
\textup{(ii)} $\overline{D_{2}}\cap U_{2}= U_{2}$ and $U_{2}\backslash D_{2}$ contains a maximally totally-real submanifold of $V$, where $V$ is a subdomain of $U_{2}$.\\
\textup{(iii)} $D_{2}=U_{2}\backslash \{z\in U_{2}:\rho(z)=0\}$, where $\rho(z)$ is a non-degenerate $C^1$ smooth function of $U_{2}$. Besides, at least one of the non-empty open subsets $\{z\in U_{2}:\rho(z)<0\}$ and $\{z\in U_{2}:\rho(z)>0\}$ is connected.
\end{definition}

\begin{remark}\label{R:RS}
Condition (i) in Definition \ref{def} was required in the original Remmert-Stein theorem (\cite{Remmert}), while condition (ii) was considered in \cite{JJ}.
\end{remark}

Inspired by the Remmert-Stein Theorem, we have the following slightly more general result. For completeness, we include a proof here.

\begin{lemma}\label{Levi}
Let $D$ be a domain in $\mathbb{C}^{m}$, $m\ge 2$. If a point $P\in \partial{D}$ has the R-S property, then there is no proper holomorphic map from $D$ to $B^n$, $n\ge 2$.
\end{lemma}
\begin{proof}
By Remark \ref{R:RS}, we only consider condition (iii) of the R-S property. Suppose that there is a proper holomorphic map $f=(f_{1}, \ldots, f_{m})$ of $D$ into $B^n$. Write $z$ as $(\xi,\omega)$, $\xi \in \mathbb{C}^{n_{1}}$, $\omega \in \mathbb{C}^{n_{2}}$.

Assume that $\tilde{D}_2=\{z\in U_{2}:\rho(z)>0\}$ is connected. Take a sequence $\omega_{\nu} \in \tilde{D}_2$ and suppose that $\omega_{\nu} \rightarrow  \omega \in \partial{\tilde{D}_2}\cap U_{2}$. For $j=1, \ldots, m$, the functions $\xi \mapsto f_{j}\left(\xi, \omega_{\nu}\right)$ define holomorphic functions $\phi_{j,\nu}$ on $D_{1}$, with $\ds \sum\left|\phi_{j, \nu}\right|^{2}$ bounded. By Montel's theorem, passing to a subsequence if necessary, we may assume that $\phi_{j,\nu} \rightarrow \phi_{j}$ uniformly on compact subsets of $D_{1}$.

Write $f\left(\xi, \omega_{\nu}\right)=\left(\phi_{1, \nu}(\xi), \ldots, \phi_{m, \nu}(\xi)\right)$ and $\Phi(\xi)=\left(\phi_{1}(\xi), \ldots, \phi_{m}(\xi)\right)$. For any $\xi \in D_{1}$, $\left\{\left(\xi, \omega_{\nu}\right)\right\}$ has no limit point in $D$, since $\left\{f\left(\xi, \omega_{\nu}\right)\right\}$ has no limit point in $B^n$. Hence, $\Phi(\xi)\in \partial B^n$. As $\partial B^n$ does not contain any germs of non-trivial complex-analytic curves, it does not contain non-trivial analytic disks. This implies that
$$\frac{\partial \phi_{j}}{\partial \xi_{p}}\equiv 0,\ \ \ p=1, \ldots, n_{1}.$$
Thus, by Weierstrass's theorem, we have
$$\frac{\partial f_{j}\left(\xi, \omega_{\nu}\right)}{\partial \xi_{p}} \longrightarrow \frac{\partial \phi_{j}}{\partial \xi_{p}}=0.$$
Therefore, for fixed $\xi \in D_{1}$, the function
$$\omega \longmapsto\left\{
\begin{array}{ll}
\ds \frac{\partial f_{j}(\xi,\omega)}{\partial \xi_{p}}, & \omega \in \tilde{D}_2\cap U_{2}, \\
0, & \omega \in U_{2}\backslash \tilde{D}_2,
\end{array}
\right.$$
is holomorphic on $U_{2}$ by Rado's theroem. On the other hand, $U_{2}\backslash \overline{\tilde{D}_2}$ is non-empty, so we have $\frac{\partial f_{j}(\xi, \omega)}{\partial \xi_{p}} \equiv 0$ on $D_{1} \times \tilde{D}_2$. As $D$ is connected, we conclude that $\frac{\partial f_{j}(\xi, \omega)}{\partial \xi_{p}} \equiv 0$ on $D$.

For $\omega_{0} \in \tilde{D}_2$, let $C_{\omega_{0}}$ be the connected component of $\{(z,w)\in D:\ w=\omega_{0}\}$ which contains $D_{1} \times\left\{\omega_{0}\right\}$. Then the map $f$ is constant on $C_{\omega_{0}}$. It is clear that  $C_{\omega_{0}}$ is not compact in $D$. Therefore $f^{-1}(f(\xi_{0},\omega_{0}))$ is not compact for $(\xi_{0},\omega_{0})\in D_{1} \times \tilde{D}_2$, which contradicts the fact that $f$ is a proper map.
\end{proof}

Now we apply Lemma \ref{Levi} to study the squeezing function and Fridman invariant of some special domains.

\begin{theorem}\label{Leviflat}
Let $D_{i}\subset \mathbb{C}^{n_{i}}$, $i=1,2$, be domains with dense $C^{1}$ boundary, i.e. for all $p \in \partial D_{i}$ and for all $\delta>0$ there exists $q \in B^{n_{i}}(p, \delta)$ such that $\partial D_{i}$ is $C^{1}$ smooth near $q$. Let $D=D_{1} \times D_{2}$. If $D\subset \subset B^{n}$ and $D':=B^{n} \backslash \overline{D}$ is connected, then we have
$$s_{D'}(z)=\tilde{e}_{D'}(z)=d_{c_{B^n}}^{\partial D}(z),\ \ \ z\in D'.$$
\end{theorem}
\begin{proof}
If $n_{1}+n_{2}<n$ then, by Theorem \ref{ati}, we have the conclusion. Thus, we will assume that $n_{1}+n_{2}=n$.

Let $z \in D'$. Similar as in the proof of Theorem \ref{KB}, it is easy to see that $s_{D'}(z)\ge d_{c_{B^n}}^{\partial D}(z)$. Thus it suffices to show that $\tilde{e}_{B^{n}\backslash \overline{D}}(z) \leq d_{c_{B^n}}^{\partial D}(z)$.

Let $Q \in \partial D$ such that $d_{c_{B^n}}^{\partial D}(z)=\tanh \left[c_{B^n}(z, Q)\right]$. Assume that $\tilde{e}_{D'}(z)>d_{c_{B^n}}^{\partial D}(z)$. Then, there exist $\delta>0$ and a holomorphic embedding $g: B^n \rightarrow D'$ with $g(0)=z$ such that $B^{n}(Q,\delta)\backslash \overline{D} \subset g(B^n)$.

Since $D_{1}$ and $D_{2}$ have dense $C^1$ boundaries, without loss of generality, we can find $P=(p_{1},p_{2})\in B^{n}(Q,\delta)\cap \partial D \cap \partial{g(B^n)}$, such that $p_{1}\in D_{1}$ and $\partial D_{2}$ is $C^1$ smooth near $p_{2}\in \partial D_{2}$.

Let $\rho(z)$ be a local defining function on some neighborhood $U_2$ of $p_{2}$ in $\mathbb{C}^{n_{2}}$ such that $D_{2}\cap U_2 =\{z\in U:\rho(z)<0\}$. Then there exists a connected neighborhood $V_2\subset U_2$ of $p_{2}$ and a connected neighborhood $V_{1}\subset D_{1}$ of $p_{1}$ such that (1) $V_2\backslash \overline{D}_{2}$ is connected; (2) $V_{1}\times V_2\subset B^{n}(Q,\delta)$. Let $G_{2}=V_2\backslash \overline{D_{2}}$ and $W_{2}=\{z\in V_2:\rho(z)>-\delta_{1}\}$, $\delta_{1}>0$ sufficiently small. It is easy to see that $(V_{1}\times W_{2})\cap g(B^n)=V_{1}\times G_{2}$, $G_{2}\subset W_{2}$, and $W_{2}\backslash \overline{G_{2}}$ is non-empty. Hence $P=(p_{1},p_{2})\in \partial g(B^n)$ has R-S property (i).

By Lemma \ref{Levi}, there is no proper mapping from $g(B^n)$ to $B^n$, which contradicts the fact that $g(z)$ is an injective mapping from $B^n$ to $g(B^n)$.
\end{proof}

\begin{theorem}\label{T:ii}
Let $D_{1}$ be a domain in $\mathbb{C}^{n_{1}}$ and $D_{2}=\{z\in \mathbb{C}^{n_{2}}:\ \Re(z_{j})=0,\ 1\leq j\leq n_{2}\}\cap B^{n_{2}}(0,\frac{1}{2})$, with $n_{1}\ge 1$ and $n_{2}\ge 2$. Denote $K=\overline{D_{1}\times D_2}$. If $K\subset\subset B^n$ with $n=n_{1}+n_{2}$ and $D:=B^n\backslash K$ is connected, then we have
$$s_D(z)=\tilde{e}_D(z)=d_{c_{B^n}}^K(z),\ \ \ z\in D.$$
\end{theorem}
\begin{proof}
Let $z\in D$. By Theorem \ref{KB} and Lemma \ref{es}, it suffices to show that $\tilde{e}_D(z)\le d_{c_{B^n}}^K(z)$.

Let $Q\in K$ such that $d_{c_{B^n}}^K(z)=\tanh[c_{B^n}(z,Q)]$. Assume that $\tilde{e}_D(z)>d_{c_{B^n}}^K(z)$. Then, there exist $\delta>0$ and a holomorphic embedding $g: B^n \rightarrow K$ with $g(0)=z$ such that $B^n(Q,\delta)\backslash K \subset g(B^n)$. We can find $P=(p_{1},p_{2})\in B^n(Q,\delta)\cap \partial g(B^n)$ with $p_{1}\in D_{1}$ and $p_{2}\in D_{2}$. Thus there exist $\delta_{j}>0$, $j=1,2$, such that $B^{n_{1}}(p_{1},\delta_{1})\times B^{n_{2}}(p_{2},\delta_{2}) \subset B^n(Q,\delta)$.

Denote $U_{1}=B^{n_{1}}(p_{1},\delta_{1})$ and $U_{2}=B^{n_{2}}(p_{2},\delta_{2})\backslash \{z\in \mathbb{C}^{n_{2}}:\ \Re(z_{j})=0,\ 1\leq j\leq n_{2}\}$. Then we have
$$(B^{n_{1}}(p_{1},\delta_{1})\times B^{n_{2}}(p_{2},\delta_{2}))\cap g(B^n)=U_{1}\times U_{2}.$$

Since $n_{2}\ge 2$, $U_{2}$ is connected and $\overline{U_{2}}=B^{n_{2}}(p_{2},\delta_{2})$. Hence $P=(p_{1},p_{2})\in \partial{g(B^n)}$ has R-S property (ii).

By Lemma \ref{Levi}, there is no proper mapping from $g(B^n)$ to $B^n$, which contradicts the fact that $g(z)$ is an injective mapping from $B^n$ to $g(B^n)$.
\end{proof}

\begin{theorem}\label{rs}
Let $D_{1}$ be a domain in $\mathbb{C}^{n_{1}}$ and $D_{2}=\{z\in \mathbb{C}^{n_{2}}:\ \Re(z_{n_{2}})=0\}\cap B^{n_{2}}(0,\frac{1}{2})$. Denote $K=\overline{D_{1}\times D_2}$. If $K\subset\subset B^n$ with $n=n_{1}+n_{2}$ and $D:=B^n\backslash K$ is connected, then we have
$$s_D(z)=\tilde{e}_D(z)=d_{c_{B^n}}^K(z),\ \ \ z\in D.$$
\end{theorem}
\begin{proof}
The proof is almost the same as that of Theorem \ref{T:ii}, except that we take $U_2=B^{n_{2}}(p_{2},\delta_{2})\backslash \{z\in \mathbb{C}^{n_{2}}:\ \Re(z_{n_{2}})=0\}$ and get that
$$(B^{n_{1}}(p_{1},\delta_{1})\times B^{n_{2}}(p_{2},\delta_{2}))\cap g(B^n)=U_{1}\times U_{2}.$$
Hence $P=(p_{1},p_{2})\in \partial{g(B^n)}$ has R-S property (iii), and we get the same contradiction.
\end{proof}

\section{Examples and further results}\label{S:examples}

Based on our study in section \ref{S:invariants}, we compute the precise form of the (generalized) squeezing functions and Fridman invariants of many example domains in this section.

Denote by $B^n$ the unit ball and $\dv^n$ the unit polydisk in $\cv^n$, $n\ge 2$. First, recall that (\cite[Corollary 2.3.5]{Pflug2013})
\begin{equation}\label{E:cB}
\tanh c_{B^{n}}(a,z)=\left[1-\frac{\left(1-\|a\|^{2}\right)\left(1-\|z\|^{2}\right)}{|1-\langle z, a\rangle|^{2}}\right]^{\frac{1}{2}}.
\end{equation}
And by the product property of the Carath\'{e}odory pseudodistance (\cite[Theorem 18.2.1]{Pflug2013}), we have
\begin{equation}\label{E:cD}
\tanh c_{\mathbb{D}^n}(z,a)=\max_{1\leq i \leq n}\left|\frac{z_{i}-a_{i}}{1-\overline{a_{i}}z_{i}}\right|.
\end{equation}

By Theorem \ref{KB}, we have the following

\begin{corollary}
Let $0<r<1$ and $\mathbb{D}^{n}_{r}=\left\{z \in \mathbb{C}^{n}:\ \left|z_{i}\right|<r,\ 1\leq i\leq n \right\}$, $n\ge 2$. Then
$$s^{\mathbb{D}^n}_{\mathbb{D}^{n}\backslash \overline{\mathbb{D}^{n}_{r}}}(z)=\max_{1\leq i\leq n} \frac{\max\left\{\left|z_{i}\right|,r\right\}-r}{1-r\left|z_{i}\right|},\ \ \ z\in \mathbb{D}^{n}\backslash \overline{\mathbb{D}^{n}_{r}}.$$
\end{corollary}
\begin{proof}
From \eqref{E:cD}, it is easy to see that for $z\in \mathbb{D}^{n}\backslash \overline{\mathbb{D}^{n}_{r}}$,
$$d_{c_{\mathbb{D}^n}}^{\partial \mathbb{D}^{n}_{r}}(z)=\max_{1\leq i\leq n} \frac{\max\left\{\left|z_{i}\right|,r\right\}-r}{1-r\left|z_{i}\right|}.$$
\end{proof}

By Theorem \ref{man}, we have the following

\begin{corollary}\label{C:annuli}
$\textup{(i)}$ Let $A^{n}_{r}=\{z \in \mathbb{C}^{n}:\ 0<r<\|z\|<1\}$, $n\ge 2$. Then
$$s_{A^{n}_{r}}(z)=\tilde{e}_{A^{n}_{r}}(z)=\frac{\left\|z \right\|-r}{1-r\left\|z \right\|},\ \ \ z\in B^n\backslash \overline{A_r^n}.$$
$\textup{(ii)}$ Let $B^{n}_{r}=\left\{z\in \mathbb{C}^n:\ \|z\|<r<1\right\}$, $n\ge 2$. Then
$$s^{\mathbb{D}^n}_{\mathbb{D}^{n}\backslash \overline{B^{n}_{r}}}(z)=\tilde{e}^{\mathbb{D}^n}_{\mathbb{D}^{n}\backslash \overline {B^{n}_{r}}}(z)=\min_{w\in \partial B^{n}_{r}}\max_{1\leq i\leq n}\left|\frac{\left|z_{i}\right|-\left|w_{i}\right|}{1-\left|w_{i}\right|\left|z_{i}\right|}\right|,\ \ \ z\in \mathbb{D}^{n}\backslash \overline{B^{n}_{r}}.$$
\end{corollary}
\begin{proof}
(i) From \eqref{E:cB}, it is easy to see that for $z\in B^n\backslash \overline{A_r^n}$,
$$d_{c_{B^n}}^{\partial A_r^n}(z)=\frac{\left\|z \right\|-r}{1-r\left\|z \right\|}.$$
(ii) From \eqref{E:cD}, it is easy to see that for $z\in \mathbb{D}^{n}\backslash \overline{B^{n}_{r}}$,
$$d_{c_{\mathbb{D}^n}}^{\partial B_r^n}(z)=\min_{w\in \partial B^{n}_{r}}\max_{1\leq i\leq n}\left|\frac{\left|z_{i}\right|-\left|w_{i}\right|}{1-\left|w_{i}\right|\left|z_{i}\right|}\right|.$$
\end{proof}

\begin{remark}
From Corollary \ref{C:annuli} (i), we see that the squeezing function of the $n$-dimensional annuli $A^{n}_{r}$, $n\ge 2$, is very different from that of the $1$-dimensional annuli $A_{r}$. While $s_{A_{r}^n}(z)\rightarrow 1$ as $\|z\|\rightarrow 1$, we actually have $s_{A^{n}_{r}}(z)\rightarrow 0$ as $\|z\|\rightarrow r$. On the other hand, it is clear that $s^{\mathbb{D}^n}_{\mathbb{D}^{n}\backslash \overline {B^{n}_{r}}}(z) \rightarrow 1$, as $z\rightarrow \partial \mathbb{D}^n$.
\end{remark}

By Theorem \ref{Leviflat}, we have the following

\begin{corollary}
For $0<r<\frac{1}{\sqrt{n}}$, we have
$$s_{B^{n}\backslash \overline{\mathbb{D}^{n}_{r}}}(z)=\tilde{e}_{B^{n}\backslash \overline{\mathbb{D}^{n}_{r}}}(z)=\min_{w\in \partial \mathbb{D}^{n}_{r}}\left[1-\frac{\left(1-\|w\|^{2}\right)\left(1-\|z\|^{2}\right)}{(1-\sum ^{n}_{i=1}|z_{i}w_{i}|\:)^{2}}\right]^{\frac{1}{2}},\ \ \ z\in B^{n}\backslash \overline{\mathbb{D}^{n}_{r}}.$$
\end{corollary}

In the remainder of this section, we give several results on the (generalized) squeezing function and Fridman invariant of the punctured polydisk or the punctured ball.

\begin{theorem}\label{BP}
For any $z\in \mathbb{D}^{n}\backslash \{0\}$, $n\ge 2$, we have
$$\min\left\{\frac{\|z\|}{\sqrt{n}},\frac{1}{\sqrt{n}}\right\}\leq s_{\mathbb{D}^{n}\backslash \{0\}}(z)\leq \tilde{e}_{\mathbb{D}^{n}\backslash \{0\}}(z)\leq \min\left\{\max_{1\leq i\leq n}{\left|z_{i}\right|},\frac{1}{\sqrt{n}}\right\}.$$
In particular, $s_{\mathbb{D}^{n}\backslash \{0\}}(z)=\tilde{e}_{\mathbb{D}^{n}\backslash \{0\}}(z)=\frac{1}{\sqrt{n}}$, if $\|z\|\ge 1$.
\end{theorem}
\begin{proof}
Fix any $z\in\mathbb{D}^{n}\backslash \{0\}$. Let $\psi$ be the automorphism of $\dv^n$ such that $\psi(z)=0$ and $\psi(0)=z$. Denote $g(w)=\frac{1}{\sqrt{n}}w$, which is an injective holomorphic map from $\mathbb{D}^{n}$ into $B^n$. Let $f(w)=g\circ\psi(w)$. Clearly $f$ is an injective holomorphic map from $\mathbb{D}^{n}\backslash \{0\}$ into $B^n$ such that $f(z)=0$ and $f(0)=\frac{1}{\sqrt{n}}z$. Thus,
$$B^{n}\left(0,\min\left\{\frac{\|z\|}{\sqrt{n}},\frac{1}{\sqrt{n}}\right\}\right)\subset f(\mathbb{D}^{n}\backslash \{0\}).$$
Therefore, we have
$$s_{\mathbb{D}^{n}\backslash \{0\}}(z)\ge\min\left\{\frac{\|z\|}{\sqrt{n}},\frac{1}{\sqrt{n}}\right\}.$$

Now consider any holomorphic embedding $h: B^n \rightarrow \mathbb{D}^{n}\backslash \{0\}$ with $h(0)=z$. Note that $\ds \tanh c_{\mathbb{D}^n}(z,0)=\max_{1\leq i\leq n}{\left|z_{i}\right|}$. Thus, if $\ds \tilde{e}_{\mathbb{D}^{n}\backslash \{0\}}(z)>\max_{1\leq i\leq n}{\left|z_{i}\right|}$, then $h(B^n)$ contains a punctured neighborhood of the origin, contradicting the fact that $h(B^n)$ is pseudoconvex. Therefore, we have $\ds \tilde{e}_{\mathbb{D}^{n}\backslash \{0\}}(z)\leq \max_{1\leq i\leq n}{\left|z_{i}\right|}$.

The fact that $\ds \tilde{e}_{\mathbb{D}^{n}\backslash \{0\}}(z)\le \frac{1}{\sqrt{n}}$ follows from \cite[Proposition 2]{Alex}.
\end{proof}

\begin{theorem}
Let $\dv^\ast=\mathbb{D}\backslash\{0\}$. For any $z\in \dv^\ast\times \dv^{n-1}$, $n\ge 1$, we have
$$\frac{|z_{1}|}{\sqrt{1+(n-1)|z_{1}|^2}}\leq s_{\dv^\ast\times \mathbb{D}^{n-1}}(z)\leq \tilde{e}_{\dv^\ast\times\mathbb{D}^{n-1}}(z)\leq \min\left\{|z_{1}|,\frac{1}{\sqrt{n}}\right\}.$$
\end{theorem}
\begin{proof}
By \cite[Corollary 7.2]{Deng2012}, we know that $s_{\dv^\ast}(z)=|z|$. Then, by \cite[Theorem 3]{NTT}, we get
$$s_{\dv^\ast\times\mathbb{D}^{n-1}}(z)\ge \frac{|z_{1}|}{\sqrt{1+(n-1)|z_{1}|^2}}.$$

The argument for $\ds \tilde{e}_{\dv^\ast\times\mathbb{D}^{n-1}}(z)\leq \min\left\{|z_{1}|,\frac{1}{\sqrt{n}}\right\}$ is very similar to the proof of Theorem \ref{BP}. We omit the details.
\end{proof}

\begin{theorem}
For any $z\in B^{n}\backslash \{0\}$, $n\ge 2$, we have
$$s^{\mathbb{D}^n}_{B^n\backslash\{0\}}(z)=\tilde{e}^{\mathbb{D}^n}_{B^n\backslash\{0\}}(z)=\min\left\{\|z\|,\frac{1}{\sqrt{n}}\right\}.$$
\end{theorem}
\begin{proof}
Fix any $z\in B^n\backslash \{0\}$. Let $\psi$ be the automorphism of $B^n$ such that $\psi(z)=0$ and $\psi(0)=z$, and $\varphi$ be the automorphism of $B^n$ such that $\varphi(0)=0$ and $\varphi(z)=(\|z\|,0,\cdots,0)$. Denote by $g$ the inclusion map from $B^n$ into $\mathbb{D}^{n}$. Let $f(w)=\varphi\circ g\circ\psi(w)$. Clearly $f$ is an injective holomorphic map from $B^{n}\backslash \{0\}$ into $\dv^n$ such that $f(z)=0$ and $f(0)=(\|z\|,0,\cdots,0)$. Thus,
$$\left\{w\in \dv^n:\ |w_{j}|<\min\left\{\|z\|,\frac{1}{\sqrt{n}}\right\},\ 1\leq j\leq n \right\}\subset f(B^{n}\backslash \{0\}).$$
Therefore, we have
$$s^{\mathbb{D}^n}_{B^n\backslash\{0\}}(z)\ge \min\left\{\|z\|,\frac{1}{\sqrt{n}}\right\}.$$

Now consider any holomorphic embedding $h: \dv^n \rightarrow B^{n}\backslash \{0\}$ with $h(0)=z$. Note that $\tanh c_{B^n}(z,0)=\|z\|$. Thus, if $\tilde{e}^{\mathbb{D}^n}_{B^n\backslash\{0\}}(z)>\|z\|$, then $h(\dv^n)$ contains a punctured neighborhood of the origin, contradicting the fact that $h(\dv^n)$ is pseudoconvex. Therefore, we have $\tilde{e}^{\mathbb{D}^n}_{B^n\backslash\{0\}}(z)\le \|z\|$.

The fact that $\ds \tilde{e}^{\mathbb{D}^n}_{B^n\backslash\{0\}}(z)\le \frac{1}{\sqrt{n}}$ follows from \cite[Proposition 1]{Alex}.
\end{proof}

\section{Non-plurisubharmonic squeezing functions and Fridman invariants}\label{S:nonpsh}

The main goal of this section is to provide some domains whose (generalized) squeezing function or Fridman invariant is not plurisubharmonic.

We first consider domains $D$ with non-plurisubharmonic $s_{D}(z)$ or $\tilde{e}_{D}(z)$.

\begin{theorem}\label{T:4.1}
Let $0<r<1$, $Q=(0,0,\cdots,r)$ and $\epsilon>0$ sufficiently small such that $B^n(Q,\epsilon)\subset B^{n}$, $n\ge 2$. Denote $A=(\partial B^{n}_{r})\backslash B^n(Q,\epsilon)$ and $D=B^n\backslash A$. Then $s_{D}(z)$ and $\tilde{e}_{D}(z)$ are not plurisubharmonic.
\end{theorem}
\begin{proof}
Denote $H=\left\{z\in \mathbb{C}^n:\ z_2=z_3=\cdots=z_n=0\right\}$. Then $B_r^n\cap H\subset D\cap H$ is a disk of radius $r$ centered at the origin. By Theorem \ref{bdary}, we have
$$s_{D}(z)=\tilde{e}_{D}(z)=d_{c_{B^n}}^{A}(z),\ \ \ z\in D.$$
From \eqref{E:cB}, it is clear that
$$s_{D}(z)=\tilde{e}_{D}(z)=\frac{r-\left\|z \right\|}{1-r\left\|z \right\|},\ \ \ z\in B^{n}_{r}\cap H.$$
Thus, $s_{D}(0)$ is maximal in $B^{n}_{r}\cap H$, showing that $s_{D}(z)$ and $e_{D}(z)$ are not subharmonic in $B^{n}_{r}\cap H$, which implies that $s_{D}(z)$ and $\tilde{e}_{D}(z)$ are not plurisubharmonic.
\end{proof}

Let $0<r<R<1$. For any $z\in \partial B_r^n$, denote $p=\frac{R}{r}z$. Then, from \eqref{E:cB}, we have
$$\tanh c_{B^n}(z,p)=\frac{R-r}{1-rR}<R.$$
Since $c_{B^n}(z_{1},z_{2})$ is continuous, there exists $\delta>0$ such that $B^n(z,\delta)\subset B^n$ and $\tanh c_{B^n}(p,w)<R$ for all $w\in B^n(z,\delta)$.

\begin{theorem}\label{ball}
Let $0<r<R<1$. Choose $\delta>0$ such that for each $z\in \partial B^{n}_{r}$ and $p=\frac{R}{r}z$, we have $B^n(z,\delta)\subset B^n$ and $\tanh c_{B^n}(p,w)<R$ for all $w\in B^n(z,\delta)$. Take $\{z_{i}\}_{i=1}^m\subset \partial B^{n}_{r}$ such that $\partial B^{n}_{r}\subset \bigcup_{i=1}^{m}B^n(z_{i},\delta)$ and set $p_{i}=\frac{R}{r}z_i$. Denote
$$H_{i}=\left\{w\in \cv^n:\ \sum^{n}_{j=1}{\overline{(z_{ij}-p_{ij})}(w_{j}-p_{ij})}=0\right\},\ \ \ 1\le i\le m.$$
Denote $H=\bigcup^{m}_{i=1}H_{i}$ and $D=B^n\backslash H$. Then $s_D(z)$ and $\tilde{e}_D(z)$ are not plurisubharmonic.
\end{theorem}
\begin{proof}
By Theorem \ref{ati}, we know that $s_D(z)=\tilde{e}_D(z)=d_{c_{B^n}}^H(z)$ for all $z\in D$, and $s_{B^n\backslash H_i}(z)=\tilde{e}_{B^n\backslash H_i}(z)=d_{c_{B^n}}^{H_i}(z)$ for all $z\in B^n\backslash H_i$. From \eqref{E:cB}, one readily checks that $s_D(0)=R$. It is also clear that $s_D(z)\leq s_{B^n\backslash H_{i}}(z)$ for $z\in D$.

Let $w\in \partial B^{n}_{r}$. Since $\partial B^{n}_{r}\subset \bigcup_{i=1}^{m}B^n(z_{i},\delta)$, we have $w\in \partial B^{n}_{r}\cap B^n(z_{i},\delta)$ for some $i$. Then we get $\tanh c_{B^n}(p_{i},w)<R$, which implies that $s_D(w)\leq s_{B^n\backslash H_{i}}(w)<R=s_D(0)$. Thus, we have $s_D(w)<s_D(0)$ for all $w\in \partial B_r^n$. Clearly, this implies that $s_D(z)$ is not plurisubharmonic.
\end{proof}

Note that $B^n\backslash H$ is pseudoconvex. In fact, let $f_{i}(w)=\sum^{n}_{j=1}\overline{(z_{ij}-p_{ij})}(w_{j}-p_{ij})$, $1\le i\le m$, and $g(w)=\frac{1}{f_{1}(w)f_{2}(w)\cdots f_{m}(w)}$. It is easy to see that $g(w)\rightarrow \infty$, as $w \rightarrow H$.

\begin{corollary}\label{usq}
Let $0<r<R<1$. Choose $\delta>0$ such that for each $z\in \dv_r$ and $p=\frac{R}{r}z$, we have $\dv(z,\delta)\subset \dv$ and $\tanh c_{\dv}(p,w)<R$ for all $w\in \dv(z,\delta)$. Take $\{z_{i}\}_{i=1}^m\subset \partial \dv_{r}$ such that $\partial \dv_{r}\subset \bigcup_{i=1}^{m} \dv(z_{i},\delta)$ and set $p_{i}=\frac{R}{r}z_i$. Denote $H=\{p_i\}_{i=1}^m$ and $D=\dv\backslash H$. Then $s_D(z)$ and $\tilde{e}_D(z)$ are not subharmonic.
\end{corollary}

\begin{corollary}
Let $\{p_i\}_{i=1}^m$ be as in Corollary \ref{usq}. Let $\epsilon>0$ such that $\overline{\dv(p_{i},\epsilon)}\cap \overline{\dv(p_{j},\epsilon)}=\emptyset$ for $i\neq j$. Denote $H_{\epsilon}=\bigcup_{i=1}^{m} (\overline{\dv(p_{i},\epsilon)}\cap \partial \dv_R)$ and $D_{\epsilon}=\mathbb{D}\backslash H_{\epsilon}$. Then $s_{D_{\epsilon}}(z)$ is not subharmonic for sufficiently small $\epsilon$.
\end{corollary}
\begin{proof}
Clearly, $D_{\epsilon}\subset D$, $D_{\epsilon_{1}}\subset D_{\epsilon_{2}}$ if $\epsilon_{2}\leq \epsilon_{1}$ and $\bigcup D_{\epsilon}=D$. Thus, by \cite[Theorem 3.8]{RY}, we have $\ds \lim_{\epsilon\rightarrow 0}s_{D_{\epsilon}}(z)=s_{D}(z)$
uniformly on compact subsets of $D$.

Since the squeezing function is continuous (\cite{Deng2012}) and $s_{D}$ is not subharmonic, there exist $z_{0}\in D$ and $\rho>0$ with $\dv(z_{0},\rho)\subset\subset D$ such that
$$s_{D}(z_{0}) > \frac{1}{2 \pi} \int_{-\pi}^{\pi} s_{D}\left(z_{0}+\rho e^{i\theta}\right) d \theta.$$
Hence, for sufficiently small $\epsilon$, we have
$$s_{D_{\epsilon}}(z_{0}) > \frac{1}{2 \pi} \int_{-\pi}^{\pi} s_{D_{\epsilon}}\left(z_{0}+\rho e^{i\theta}\right) d \theta.$$
This implies that $s_{D_{\epsilon}}(z)$ is not subharmonic.
\end{proof}

\begin{remark}
Using the Riemann mapping theorem, one can prove that $D_{\epsilon}$'s in the above corollary are holomorphically equivalent to bounded domains with smooth boundary (cf. \cite{Ahl}). Hence there exist bounded domains with smooth
boundary whose squeezing function is not subharmonic.
\end{remark}

\begin{example}
Let $0<r<R<1$. Set $\delta=r\sqrt{2-2\cos{\frac{2\pi}{m}}}$, and choose $m$ big enough such that $\delta<\frac{r-rR^2}{1+R^2}$ and $r+\delta<1$. Take $z_{j}=re^{i\frac{2j\pi}{m}}$ and $p_{j}=Re^{i\frac{2j\pi}{m}}$, $1\leq j\leq m$. Since $|z_{j+1}-z_{j}|=r\sqrt{2-2\cos{\frac{2\pi}{m}}}$, $1\leq j\leq m-1$, it follows that $\partial \mathbb{D}_{r}\subset \bigcup_{j=1}^{m}\mathbb{D}(z_j,\delta)$. On the other hand, we have
$$\tanh c_{\dv}(p_j,w)=\left|\frac{w-p_{j}}{1-\overline{p_{j}}w}\right|\leq\frac{|w-z_{j}|+|z_{j}-p_{j}|}{1-R(r+\delta)}\leq \frac{\delta+R-r}{1-R(r+\delta)}< R,$$
for any $w\in \mathbb{D}(z_{j},\delta)$, $1\leq j\leq m$. Hence, Corollary \ref{usq} applies in this setting.
\end{example}

By Theorem \ref{BP}, we have $s_{\mathbb{D}^{n}\backslash \{0\}}(z)=\frac{1}{\sqrt{n}}$, for any $z\in \dv^n\backslash \{0\}$ with $\|z\|\ge 1$. Take $\|z\|<1$ with $|z_{1}|=|z_{2}|=\cdots=|z_{n}|$. Then we have $s_{\mathbb{D}^{n}\backslash \{0\}}(z)=|z_{i}|<\frac{1}{\sqrt{n}}$. Hence $s_{\mathbb{D}^n\backslash \{0\}}(z)$ is not plurisubharmonic. On the other hand, by Theorem \ref{ati}, $\ds s^{\mathbb{D}^n}_{\mathbb{D}^n\backslash \{0\}}(z)=\max_{1\leq j \leq n}{|z_{j}|}$ is plurisubharmonic. Therefore, It is of interest to consider domains with non-plurisubharmonic generalized squeezing function with model domain $\Omega=\mathbb{D}^n$.

The next result is similar to Theorem \ref{T:4.1}. However, since $\partial \dv_r^n$ is Levi-flat, we can only use Theorem \ref{KB} to study the generalized squeezing function.

\begin{theorem}
Let $0<r<1$, $Q=(0,0,\cdots,r)$ and $\epsilon>0$ sufficiently small such that $B^n(Q,\epsilon)\subset \dv^{n}$, $n\ge 2$. Denote $A=(\partial \dv^{n}_{r})\backslash B^n(Q,\epsilon)$ and $D=\dv^n\backslash A$. Then $s_{D}^{\dv^n}(z)$ is not plurisubharmonic.
\end{theorem}
\begin{proof}
Denote $H=\left\{z\in \mathbb{C}^n:\ z_2=z_3=\cdots=z_n=0\right\}$. Then $\dv_r^n\cap H\subset D\cap H$ is a disk of radius $r$ centered at the origin. By Theorem \ref{KB}, we have
$$s_{D}^{\dv^n}(z)=d_{c_{\dv^n}}^{A}(z),\ \ \ z\in D.$$
From \eqref{E:cD}, it is clear that
$$s_{D}^{\dv^n}(z)=\frac{r-|z_1|}{1-r|z_1|},\ \ \ z\in \dv^{n}_{r}\cap H.$$
Thus, $s_{D}^{\dv^n}(0)$ is maximal in $\dv^{n}_{r}\cap H$, showing that $s_{D}^{\dv^n}(z)$ is not subharmonic in $\dv^{n}_{r}\cap H$, which implies that $s_{D}^{\dv^n}(z)$ is not plurisubharmonic.
\end{proof}

The following result is similar to Theorem \ref{ball}.

\begin{theorem}\label{polydisk}
Let $0<r<R<1$ and $H=\left\{z\in \mathbb{C}^n:\ z_2=z_3=\cdots=z_n=0\right\}$. Choose $\delta>0$ such that for each $z\in \partial \dv^{n}_{r}\bigcap H$ and $p=\frac{R}{r}z$, we have $B^n(z,\delta)\subset \dv^n$ and $\tanh c_{\dv^n}(p,w)<R$ for all $w\in B^n(z,\delta)$. Take $\{z_{i}\}_{i=1}^m\subset \partial \dv^{n}_{r}\bigcap H$ such that $\partial \dv^{n}_{r}\bigcap H\subset \bigcup_{i=1}^{m} B^n(z_{i},\delta)$ and set $p_{i}=\frac{R}{r}z_i$. Denote
$$H_{i}=\left\{w\in \cv^n:\ w_{1}=p_{i1}\right\},\ \ \ 1\le i\le m.$$
Denote $A=\bigcup^{m}_{i=1}H_{i}$ and $D=\dv^n\backslash A$. Then $s_D^{\dv^n}(z)$ and $\tilde{e}_D^{\dv^n}(z)$ are not plurisubharmonic.
\end{theorem}
\begin{proof}
By Theorem \ref{ati}, we know that $s_D^{\dv^n}(z)=\tilde{e}_D^{\dv^n}(z)=d_{c_{\dv^n}}^A(z)$ for all $z\in D$, and $s_{B^n\backslash H_i}^{\dv^n}(z)=\tilde{e}_{B^n\backslash H_i}^{\dv^n}(z)=d_{c_{\dv^n}}^{H_i}(z)$ for all $z\in \dv^n\backslash H_i$. From \eqref{E:cD}, one readily checks that $s_D^{\dv^n}(0)=R$. It is also clear that $s_D^{\dv^n}(z)\leq s_{\dv^n\backslash H_{i}}^{\dv^n}(z)$ for $z\in D$.

Let $w\in \partial \dv^{n}_{r}\bigcap H$. Since $\partial \dv^{n}_{r}\bigcap H\subset \bigcup_{i=1}^{m} B^n(z_{i},\delta)$, we have $w\in \partial B^{n}_{r}\cap B^n(z_{i},\delta)$ for some $i$. Then we get $\tanh c_{\dv^n}(p_{i},w)<R$, which implies that $s_D^{\dv^n}(w)\leq s_{B^n\backslash H_{i}}^{\dv^n}(w)<R=s_D^{\dv^n}(0)$. Thus, we have $s_D^{\dv^n}(w)<s_D^{\dv^n}(0)$ for all $w\in \partial \dv_r^n\bigcap H$. Clearly, this implies that $s_D^{\dv^n}(z)$ is not plurisubharmonic.
\end{proof}

\end{document}